\documentclass[11pt]{article}
\usepackage[utf8]{inputenc}
\usepackage[T1]{fontenc}
\usepackage{microtype}
\usepackage[dvips]{graphicx}
\usepackage{xcolor}
\usepackage{times}
\usepackage{amsmath}
\usepackage{mathtools}
\usepackage{amstext}
\usepackage{amsfonts}
\usepackage{amssymb}
\usepackage{amsthm}
\usepackage{verbatim}
\usepackage{hhline}
\usepackage{tikz}
\usepackage{faktor}
\usepackage{mathrsfs}
\usepackage{mathtools}
\usepackage{listings}
\usepackage{framed}
\usepackage{scrextend}
\usepackage{multicol}
\usepackage{enumerate}
\usepackage{caption}
\usepackage{subcaption}
\usepackage{float}
\usepackage{hyperref}
\usepackage[margin = 1in]{geometry}

\newtheorem{theorem}{Theorem}
\newtheorem{definition}{Definition}

\newtheorem{lemma}[theorem]{Lemma}
\newtheorem{proposition}[theorem]{Proposition}
\newtheorem{corollary}[theorem]{Corollary}
\newtheorem{problem}{Problem}

\newtheorem{conjecture}{Conjecture}

\newcommand{\E}{\mathbb{E}}

\newcommand{\numkmols}[2]{L^{(#1)}(#2)}

\DeclarePairedDelimiter{\parens}{(}{)}
\DeclarePairedDelimiter{\set}{\{}{\}}
\DeclarePairedDelimiter{\brackets}{[}{]}
\DeclarePairedDelimiter{\size}{|}{|}

\newcommand{\dif}{\mathop{}\!\mathrm{d}}

\title{Enumerating extensions of mutually orthogonal Latin squares}
\author{Simona Boyadzhiyska\thanks{Institut f\"ur Mathematik, Freie Universit\"at Berlin, 14195 Berlin}~
        \thanks{E-mail: {\tt s.boyadzhiyska@fu-berlin.de}. Research supported by the Deutsche Forschungsgemeinschaft (DFG) Graduiertenkolleg ``Facets of Complexity'' (GRK 2434).}
    \and
        Shagnik Das\footnotemark[1]~
        \thanks{E-mail: {\tt shagnik@mi.fu-berlin.de}.  Research supported in part by GIF grant G-1347-304.6/2016 and by the Deutsche Forschungsgemeinschaft (DFG) - project 415310276.}
    \and
        Tibor Szab\'o\footnotemark[1]~
        \thanks{E-mail: {\tt szabo@math.fu-berlin.de}.  Research supported in part by GIF grant G-1347-304.6/2016.}}

\begin{document}
\maketitle 

\begin{abstract}
Two $n \times n$ Latin squares $L_1, L_2$ are said to be \emph{orthogonal} if, for every ordered pair $(x,y)$ of symbols, there are coordinates $(i,j)$ such that $L_1(i,j) = x$ and $L_2(i,j) = y$. A \emph{$k$-MOLS} is a sequence of $k$ pairwise-orthogonal Latin squares, and the existence and enumeration of these objects has attracted a great deal of attention.

Recent work of Keevash and Luria provides, for all fixed $k$, log-asymptotically tight bounds on the number of $k$-MOLS. To study the situation when $k$ grows with $n$, we bound the number of ways a $k$-MOLS can be extended to a $(k+1)$-MOLS. These bounds are again tight for constant $k$, and allow us to deduce upper bounds on the total number of $k$-MOLS for all $k$. These bounds are close to tight even for $k$ linear in $n$, and readily generalise to the broader class of gerechte designs, which include Sudoku squares.
\end{abstract}

\section{Introduction}\label{sec:intro}

Latin squares have a long and storied history in combinatorics, sharing connections to several other areas of mathematics and enjoying applications in statistics and experimental design.  In particular, orthogonal Latin squares are equivalent to many other classical structures in design theory, and their study dates back to Euler.  In this paper we shall answer questions concerning the enumeration of orthogonal Latin squares, but we first present some of the relevant background.

\subsection{Background and related work}\label{sec:background}

We begin by recalling the definition of a Latin square, more in an effort to present our notation than in belief that you, the reader, do not know what a Latin square is.

\begin{definition}\rm \label{def:latinsq}
A {\em Latin square of order $n$} is an $n\times n$ matrix with entries in $[n]$ such that each $x\in[n]$ appears exactly once in every row and in every column.
\end{definition}

It is not difficult to see that Latin squares exist for all $n$; indeed, a rich class of constructions are given by the Cayley tables of groups.  We refer the reader to~\cite{van2001course} for more definitions, results and proofs related to Latin squares, noting only that the number of Latin squares is log-asymptotically given by
\begin{align}\label{eq:numlatsq}
    L(n) = \parens*{(1+o(1))\frac{n}{e^2}}^{n^2}.
\end{align}
Ryser~\cite{ryser1969} showed that the lower bound follows from Van der Waerden's conjecture on permanents of matrices, which was famously later proven by Egorychev~\cite{egorychev1981} and Falikman~\cite{falikman1981}.  The upper bound is also closely related to permanents, as it is a consequence of Br\'egman's Theorem~\cite{bregman1973} (see \cite[Chapter 17]{van2001course} for details).

\smallskip

In this paper, we will be concerned with orthogonal Latin squares.

\begin{definition}\rm \label{def:orthlatinsq}
Two Latin squares $L, L'$ of order $n$ are said to be {\em orthogonal} if, for all pairs $(x,y) \in [n]^2$, there exist unique $i,j\in [n]$ such that $L(i,j) = x$ and $L'(i,j) = y$. In this case, $L'$ is said to be an {\em orthogonal mate} of $L$. 

A $k$-tuple of Latin squares $(L_1,\dots, L_k)$ forms a system of {\em $k$ mutually orthogonal Latin squares}, or a {\em $k$-MOLS}, if for all $1\leq i < j\leq k$, the squares $L_i$ and $L_j$ are orthogonal. 
\end{definition}

Orthogonal Latin squares have proven interesting from both a theoretical and a practical point of view: they are related to various other classes of structures in design theory, some of which we shall encounter later, while having applications in many real-world problems. In light of these applications, the early research in this area concerned the existence of $k$-MOLS. In particular, there was much interest in the maximum size of a set of mutually orthogonal Latin squares; that is, the function $N(n) = \max \{ k : \textrm{a } k\textrm{-MOLS of order } n \textrm{ exists} \}$.

It is well-known that $N(n) \le n-1$.  Indeed, suppose $(L_1,\dots,L_k)$ is a $k$-MOLS. Observing that orthogonality is preserved under permutations of the symbols within each square, we may assume that the first row of each square is $[1,2,\dots, n]$. Considering the entries in position $(2,1)$, we find that all $L_i(2,1)$ must be distinct, by orthogonality, and different from 1, since the $L_i$ are Latin squares.  Hence $k \le n-1$.

One can further prove that one has equality if and only if a projective plane of order $n$ exists. This shows that the precise determination of the function $N(n)$ is likely to be difficult, as the existence of projective planes for orders $n$ that are not prime powers is a longstanding open problem. Still, several polynomial lower bounds on $N(n)$ with ever-improving exponents appear in the literature~\cite{CES1960,rogers1964,wilson1974}, with the largest one due to Lu~\cite{lu1985}, who proved $N(n) = \Omega(n^{1/14.3})$.

\smallskip

Given that large sets of mutually orthogonal Latin squares exist, it is natural to extend~\eqref{eq:numlatsq} and enumerate $k$-MOLS for $k \ge 2$. Early work in this direction was undertaken by Donovan and Grannell~\cite{donovan2013transversal}, who constructed many $k$-MOLS, and also sought to bound the number of orthogonal mates a Latin square can have. An important component of their argument is an upper bound on number of transversals in a Latin square, where a transversal is a selection of $n$ cells from the square, with no two sharing the same row, column or symbol. Taranenko~\cite{taranenko2015transversals} later proved a sharp upper bound on the number of transversals in a Latin square, which, when used in Donovan and Grannell's proof, shows that a Latin square can have at most 
\begin{equation} \label{eq:donovangrannell}
    \parens*{(1 + o(1)) \frac{n}{e^{2 + 1/e}}}^{n^2}
\end{equation}
orthogonal mates. Coupled with~\eqref{eq:numlatsq}, this can be used to give upper bounds on the number of pairs of orthogonal Latin squares and, more generally, the number of $k$-MOLS (since in a $k$-MOLS $(L_1, \hdots, L_k)$, the Latin squares $L_2, \hdots, L_k$ must all be orthogonal mates of $L_1$).

More recently, tight bounds on the number of $k$-MOLS follow from the breakthroughs of Luria~\cite{luria2017new} and Keevash~\cite{keevash2018coloured}. Through an elegant entropic argument, Luria gives a general upper bound on the number of perfect matchings a regular $r$-uniform hypergraph can have.  Assuming certain pseudorandom conditions, Keevash provides a matching lower bound, coupling randomized constructions with the use of absorbers. When applied to the enumeration of $k$-MOLS, their theorems imply the following result, where we denote by $\numkmols{k}{n}$ the number of $k$-MOLS of order $n$.

\begin{theorem}[Luria, 2017, and Keevash, 2018] \label{thm:luriakeevash}
For every fixed $k \in \mathbb{N}$, the number of $k$-MOLS of order $n$ is
\begin{equation} \label{eq:luriakeevash}
    \numkmols{k}{n} = \parens*{(1 + o(1)) \frac{n^k}{e^{\binom{k+2}{2} - 1}}}^{n^2}.
\end{equation}
\end{theorem}

\subsection{Results}\label{sec:results}

The one drawback of Theorem~\ref{thm:luriakeevash} is that both the lower and upper bounds in \eqref{eq:luriakeevash} require $k$ to be fixed as $n$ tends to infinity.  In this paper we seek upper bounds that hold when $k$ grows with $n$.  We combine the approach of Donovan and Grannell~\cite{donovan2013transversal} with the method of Luria~\cite{luria2017new}, using entropy to bound the number of ways of extending a $k$-MOLS by adding an additional Latin square.

Before presenting our upper bound, let us discuss a lower bound for this number of extensions. Since every $(k+1)$-MOLS contains a $k$-MOLS as a prefix, Theorem~\ref{thm:luriakeevash} implies that, for fixed $k \in \mathbb{N}$, the average number of extensions of a $k$-MOLS to a $(k+1)$-MOLS is at least
\begin{equation} \label{eq:extensionlowerbound}
    \frac{L^{(k+1)}(n)}{L^{(k)}(n)} = \parens*{(1+o(1)) \frac{n}{e^{k+2}}}^{n^2}.
\end{equation}

This clearly gives a lower bound for the maximum number of such extensions. In the following theorem, we provide an upper bound that is valid for all $k$. (Throughout this paper, all logarithms are to the base $e$.)

\begin{theorem}\label{thm:numberofextensions}
For $0 \leq k \leq n-2$, the logarithm of the number of ways to extend a $k$-MOLS of order $n$ to a $(k+1)$-MOLS is at most
\[ n^2 \int_0^1\log(1+(n-1)t^{k+2}) \dif t. \]
\end{theorem}

We will estimate the value of this integral in Lemma~\ref{lem:integral}. As a corollary, combining Theorem~\ref{thm:numberofextensions} with~\eqref{eq:extensionlowerbound} allows us to determine the number of extensions of a $k$-MOLS of fixed size. In particular, setting $k = 1$ bounds the number of orthogonal mates a Latin square can have, sharpening the bound in~\eqref{eq:donovangrannell}.

\begin{corollary} \label{cor:asymptoticextension}
For every fixed $k \in \mathbb{N}$, the maximum number of ways to extend a $k$-MOLS of order $n$ to a $(k+1)$-MOLS is
\[ \parens*{(1 + o(1)) \frac{n}{e^{k+2}}}^{n^2}. \]
\end{corollary}

As previously stated, our primary goal is to bound the number of $k$-MOLS when $k$ grows with $n$. We can do so by building the $k$-MOLS one Latin square at a time, using Theorem~\ref{thm:numberofextensions} to bound the number of choices at each step. In this way we can recover the upper bound of Theorem~\ref{thm:luriakeevash} when $k$ is constant, but the main novelty of this paper is the following extension to larger values of $k$.

\begin{corollary}\label{cor:numsets}
As $n\rightarrow \infty$,
\begin{enumerate}[(i)]
    \item {{$\log \numkmols{k}{n} \leq \parens*{k \log n - \binom{k+2}{2} + 1 + k^2 n^{-1/(k+2)}}n^2$ \hfill} if $k = o(\log n)$,}
    \item {{$\log \numkmols{k}{n} \leq \parens*{c(\beta) + o(1)} k n^2 \log n$ \hfill} if $k = \beta \log n$, for fixed $\beta > 0$,}
    \item {{$\log \numkmols{k}{n} \leq \parens*{\tfrac12 + o(1)} \parens*{\log k - \log \log n} n^2 \log^2 n$ \hfill} if $k = \omega(\log n)$,}
\end{enumerate}
where in (ii) we define $c(\beta) = 1 - \beta^{-1} \int_0^{\beta} x (1 - e^{-1/x}) \dif x \in [0,1]$.
\end{corollary}

Note that we trivially have $\numkmols{k}{n} \le L(n)^k$, which, in light of~\eqref{eq:numlatsq}, gives the upper bound \linebreak $\log \numkmols{k}{n} \le k n^2 \log n$. Corollary~\ref{cor:numsets} provides a significant improvement over this trivial bound. Furthermore, part (i) shows that the upper bound from Theorem~\ref{thm:luriakeevash} is valid whenever $k = o \parens*{\tfrac{\log n}{\log \log n}}$.

It is well-known that mutually orthogonal Latin squares are equivalent to many other combinatorial structures such as transversal designs, nets and orthogonal arrays, while also being related to certain error correcting codes and affine and projective planes (see \cite{van2001course}), and so our results give upper bounds for the number of structures in each of these classes. In fact, we shall prove Theorem~\ref{thm:numberofextensions} for the more general class of gerechte designs (see Theorem~\ref{thm:numberextensions}), which allow us to, for instance, bound the number of sets of mutually orthogonal Sudoku squares. We discuss this particular extension further in our concluding remarks.

\subsection{Organization}

The remainder of this paper is organized as follows. In Section~\ref{sec:tools}, we introduce gerechte designs and discuss an equivalent formulation that will be more convenient for our proof. We also review some basic notions about our main tool, entropy.  Following that, we prove our results in Section~\ref{sec:mainresult}. We then provide explicit constructions of Latin squares with many orthogonal mates in Section \ref{sec:constructions}, and close with some further remarks and open problems in Section~\ref{sec:conclusion}.


\section{Designs and tools}\label{sec:tools}

In this section we will introduce the frameworks of gerechte designs and orthogonal arrays, in which we will prove a generalization of Theorem~\ref{thm:numberofextensions}. We will also review some definitions and results regarding entropy that we shall require in our proofs.

\subsection{Gerechte designs}\label{sec:gerechte}

Gerechte designs, defined below, are a special class of Latin squares introduced by Behrens in \cite{behrens1956feldversuchsanordnungen}.

\begin{definition}\label{def:gerechtedesigns}\rm
Let $[n]^2 = R_1\sqcup\dots\sqcup R_n$ be a partition of $[n]^2$ into $n$ regions $R_i$ such that $|R_i| = n$ for all $i\in [n]$. A {\em gerechte design of order $n$} with respect to this partition is a Latin square with the additional property that each symbol appears exactly once in each region $R_i$.
\end{definition}

There are several natural examples of gerechte designs. For instance, if one takes the regions to be the $n$ rows (or columns) of the $n \times n$ grid, a gerechte design is simply a Latin square. If $n = m^2$, and one partitions the grid into $n$ subsquares of dimension $m \times m$, the corresponding gerechte designs are known as \emph{Sudoku squares of order $n$}. Finally, given a Latin square $L$, define the regions $R_t = \{(i,j) : L(i,j) = t\}$ for all $t \in [n]$. A gerechte design with respect to this partition is an orthogonal mate of $L$.

\smallskip

It is natural to study orthogonality between Latin squares that are gerechte designs with respect to the same partition and, more generally, to consider systems of mutually orthogonal gerechte designs. Bailey, Cameron and Connelly~\cite{bailey2008gerechte} generalized the function $N(n)$ to the setting of gerechte designs, giving upper bounds on the size of a set of mutually orthogonal gerechte designs that are tight for some orders $n$.

The counting questions concerning Latin squares discussed in the introduction can also be generalized to gerechte designs, and our method will allow us to derive bounds in this broader setting. For this, note that an $n\times n$ square with entries in $[n]$ is a Latin square if and only if it is orthogonal (in the sense of Definition \ref{def:orthlatinsq}) to the square $S_n$, given by $S_n(i,j) = i$ for all $i,j\in[n]$, and its transpose. Similarly, it is not difficult to show that an $n\times n$ square with entries in $[n]$ is a gerechte design with respect to the regions $R_1, \dots, R_n$ if and only if it is orthogonal to the squares $S_n$, $S_n^T$, and $B$, where $B$ is given by $B(i,j) = t$ if $(i,j)\in R_t$. Note that, while the squares $S_n$ and $S_n^T$ are orthogonal to each other, the square $B$ need not be orthogonal to either (that is, $B$ need not be a Latin square). 

\subsection{Orthogonal arrays and nearly orthogonal arrays}\label{sec:oas}

When adding a square to a set of mutually orthogonal gerechte designs, we need to ensure three properties: that it is a Latin square, that it respects the regions of the design, and that it is orthogonal to the previous squares. For our proof, it will be helpful to use an equivalent but more symmetric formulation of mutually orthogonal gerechte designs, where these three properties all take the same form. We begin in the setting of mutually orthogonal Latin squares.

\begin{definition}\rm\label{def:oa}
Let $x,y$ be vectors in $[n]^{n^2}$. We say that $x$ and $y$ are {\em orthogonal} if, for all pairs $(s,t) \in [n]^2$, there exists a unique index $\ell$ such that $x_\ell = s$ and $y_\ell = t$.
An {\em orthogonal array} $OA(n,d)$ is an $n^2 \times d$ array $A$ with entries in $[n]$ such that all pairs of its columns are orthogonal.
\end{definition}
We note that in the literature orthogonal arrays are often defined more generally and Definition \ref{def:oa} describes what is known as an {\em orthogonal array with strength two and index one}. For the sake of simplicity, we omit the general definition and refer the reader to \cite{hedayat1999orthogonal} for more about orthogonal arrays.

\smallskip

Given a $k$-MOLS $(L_1,\dots,L_k)$ of order $n$, we can construct an orthogonal array $OA(n,k+2)$ by taking, for all $(i,j) \in [n]^2$, the vectors $[i,j,L_1(i,j), L_2(i,j),\dots, L_k(i,j)]$ as rows of the orthogonal array (and ordering them lexicographically). Similarly, given an $n^2\times (k+2)$ orthogonal array $A$, we can construct a $k$-MOLS of order $n$ by setting $L_j(A(\ell,1), A(\ell,2)) = A(\ell,j+2)$ for all $1\leq \ell \leq n^2$ and $1\leq j\leq k$ (in fact, any two columns of the orthogonal array can be used to coordinatize the Latin squares; here we use the first two). Notice that distinct sequences of mutually orthogonal Latin squares correspond to distinct orthogonal arrays with first two columns $v_1 = [1,\dots, 1,2,\dots, 2,\dots, n,\dots, n]^T$ and $v_2 = [1,2,\dots, n,1,2,\dots, n, \dots, 1,2,\dots, n]^T$, and hence the number of $k$-MOLS of order $n$ is the same as the number of orthogonal arrays $OA(n,k+2)$ with first columns $v_1$ and $v_2$.

We now extend these ideas to mutually orthogonal gerechte designs. Let $v_3$ be a vector in $[n]^{n^2}$ with each integer in $[n]$ appearing $n$ times. Note that $v_3$ determines a partition of the elements of $[n^2]$ (and thus $[n]^2$, after we fix a linear ordering of this set) into $n$ equally-sized regions. From the equivalence between mutually orthogonal Latin squares and orthogonal arrays and the discussion at the end of Section~\ref{sec:gerechte}, we can conclude that an $OA(n,k+2)$, whose first two columns are $v_1$ and $v_2$, and in which all other columns are also orthogonal to $v_3$, is equivalent to $k$ mutually orthogonal gerechte designs with respect to the partition determined by $v_3$. For notational convenience, we add the column $v_3$ to the array and call the resulting structure an $n^2\times (k+3)$ {\em nearly orthogonal array}.

\begin{definition} \label{def:noa}\rm
Given $n \in \mathbb{N}$ and $d \ge 3$, a \emph{nearly orthogonal array} $NOA(n,d)$ is an $n^2 \times d$ array $A$ with symbols $[n]$ such that:
\begin{itemize}
    \item[(a)] the first column is $v_1$ and the second column is $v_2$, as defined above,
    \item[(b)] each symbol in $[n]$ appears exactly $n$ times in the third column $v_3$, and
    \item[(c)] for all $i \ge 4$, the $i$th column $v_i$ is orthogonal to all other columns in $A$.
\end{itemize}
\end{definition}

Again, it follows that the number of nearly orthogonal arrays $NOA(n,k+3)$ is equal to the number of sets of $k$ mutually orthogonal gerechte designs with respect to the partition defined by $v_3$.

\subsection{Entropy}\label{sec:entropy}
The proof of our main result is based on entropy. This method has previously given good asymptotic upper bounds for similar problems; for instance, it is used in \cite{radhakrishnan1997entropy} to prove Br\'egman's Theorem on the permanent of a matrix (which yields an asymptotically tight upper bound on the number of Latin squares), in \cite{linial2013upper} to show an upper bound on the number of Steiner triple systems, later shown to be tight in \cite{keevash2018counting}, and in \cite{glebov2016maximum} to provide a simpler proof of Taranenko's result on the maximum number of transversals in a Latin square, also shown to be tight in the same paper; see also \cite{luria2017new} for some further applications. In this section, we review some basic facts about entropy that will be used in our proof. For more on entropy, see \cite{cover1991elements}.
\smallskip

Let $X$ be a discrete random variable taking values in a given finite set $\mathcal{S}$,
and let $p(x) = \Pr[X=x]$ for all $x\in \mathcal{S}$.  
The {\em (base $e$) entropy} of $X$ is given by 
\begin{align*}
    H(X) = -\sum\limits_{x\in \mathcal{S}} p(x)\log p(x) = -\E[\log p(X)],
\end{align*}
where we adopt the convention that $0\log 0 = 0$.  The entropy of $X$ can be seen as a measure of the amount of information the random variable encodes. It is not difficult to show that 
\begin{equation} \label{eq:uniformentropy}
H(X) \leq \log |R(X)|,
\end{equation}
where $R(X) = \set{x\in \mathcal{S}: p(x) > 0}$ is the range of the random variable, with equality if and only if $X$ is uniformly distributed over $R(X)$.

This definition can be extended to multiple random variables in the natural way. We define the {\em joint entropy} of two random variables $X$ and $Y$ to be 
\begin{align*}
    H(X,Y) = - \sum\limits_{x,y} p(x,y) \log p(x,y) = -\E[\log p(X,Y)],
\end{align*}
where $p(x,y) = \Pr [X = x, Y = y]$ denotes the joint distribution of $X$ and $Y$.

The {\em conditional entropy} of $X$ given $Y$ is defined to be
\begin{align*}
    H(X|Y) = \E_Y [H(X|Y=y)] = \sum\limits_y \Pr[Y = y] H(X | Y = y).
\end{align*}
Conditional entropy gives us a way to measure how much additional information we expect to learn from $X$ once we know the value of $Y$. 
It is a simple exercise to show that the joint entropy and the conditional entropy of several random variables satisfy the following equality, known as the chain rule:
\begin{align*}
    H(X_1,\dots,X_n) = \sum\limits_{i=1}^n H(X_i | X_1,\dots, X_{i-1}).
\end{align*}

\smallskip
We end this section by outlining the basic idea behind counting proofs based on entropy. Suppose we want to obtain a bound on the size of a set $\mathcal{S}$. We sample an element $X\in \mathcal{S}$ uniformly at random. By the above discussion, we have $H(X) = \log |\mathcal{S}|$, and so an upper bound on the entropy $H(X)$ yields an upper bound on $|\mathcal{S}|$. To bound $H(X)$, we break up the random variable $X$ into simpler random variables; the chain rule then allows us to consider these new random variables one at a time.

\section{Proofs of our results}\label{sec:mainresult}

We now use the material from the previous section to prove Theorem~\ref{thm:numberofextensions} and its corollaries.

\subsection{Bounding the number of extensions}

In the language of orthogonal arrays, Theorem~\ref{thm:numberofextensions} is a statement about the number of ways to extend an orthogonal array by one column. We will in fact prove the following more general result, bounding the number of ways to extend a nearly orthogonal array by one column. Indeed, by inserting a copy of the first column in the third column (and reordering the rows if needed), one obtains a nearly orthogonal array from an orthogonal array.

\begin{theorem}\label{thm:numberextensions}
Given $n \in \mathbb{N}$ and $d \ge 3$, let $A$ be a nearly orthogonal array $NOA(n,d)$. For each row $\ell \in [n^2]$, define 
\begin{align*}
    r_\ell &= |\set{s \neq \ell: A(s,1) = A(\ell,1)\text{ and }A(s,3) = A(\ell,3)}|, \text{ and}\\
    c_\ell &= |\set{s \neq \ell: A(s,2) = A(\ell,2)\text{ and }A(s,3) = A(\ell,3)}|.   
\end{align*}
Then the logarithm of the number of ways to extend $A$ to a nearly orthogonal array with $d+1$ columns is at most
\begin{align*}
    \sum\limits_{\ell=1}^{n^2} \int_0^1\log(1+(r_\ell+c_\ell)t^{d-1}+(n-r_\ell-c_\ell-1)t^{d}) \dif t.
\end{align*}
\end{theorem}

Observe that in the gerechte design setting, for a cell $\ell \in [n]^2$, $r_\ell$ counts the number of other cells in the same row and region as $\ell$, while $c_\ell$ counts the number of cells sharing the same column and region.

Before proving Theorem~\ref{thm:numberextensions}, we quickly derive Theorem~\ref{thm:numberofextensions}.

\begin{proof}[Proof of Theorem~\ref{thm:numberofextensions}]
As previously mentioned, a Latin square is a gerechte design with respect to the partition of the cells into their rows.  A $k$-MOLS is thus equivalent to an $NOA(n,k+3)$ with $v_3 = v_1$, and an extension to a $(k+1)$-MOLS corresponds to adding a column to obtain an $NOA(n,k+4)$.

We can thus apply Theorem~\ref{thm:numberextensions} with $d = k+3 \ge 3$. For our choice of $v_3$, we have $r_\ell = n-1$ and $c_\ell = 0$ for all $\ell \in [n^2]$. Substituting in these values, the bound on the number of extensions is
\[ \sum\limits_{\ell = 1}^{n^2} \int_0^1 \log(1 + (n-1)t^{k+2}) \dif t = n^2 \int_0^1 \log(1 + (n-1)t^{k+2}) \dif t,\]
as required.
\end{proof}

We now proceed to the proof of the general theorem.

\begin{proof}[Proof of Theorem~\ref{thm:numberextensions}]
Let $A$ be as given, and let $\mathcal{S}$ denote the set of column vectors that are valid extensions for $A$. Our goal is to bound $\size{\mathcal{S}}$. We can assume $\mathcal{S} \neq \emptyset$, otherwise we are done. Let $X \in \mathcal{S}$ be chosen uniformly at random. Then $H(X) = \log \, \size{\mathcal{S}}$, and so it suffices to bound the entropy of $X$. We will expose the coordinates of $X$ one at a time, using the chain rule to express the total entropy $H(X)$ as the sum of the conditional entropies from each successive reveal.

For $\ell \in [n^2]$, we denote the $\ell$th coordinate of $X$ by $X_{\ell}$ and, given a permutation $\pi$ of $[n^2]$, we reveal the coordinates in the order $X_{\pi(1)}, X_{\pi(2)}, \hdots, X_{\pi(n^2)}$.  The chain rule then gives
\begin{align} \label{eq:entropycalc}
    \log \, \size{\mathcal{S}} = H(X) &= \sum\limits_{j=1}^{n^2}H(X_{\pi(j)} | X_{\pi(s)}: s < j) \notag \\
    &= \sum\limits_{j=1}^{n^2}\E_{(X_{\pi(s)} : s < j )}\brackets{H(X_{\pi(j)}| X_{\pi(s)} = x_{\pi(s)}: s < j)}.
\end{align}
Given $x \in [n]^{n^2}$, let $R_{\pi(j)}(\pi,x) = R(X_{\pi(j)} | X_{\pi(s)} = x_{\pi(s)} : s < j)$ denote the range of this conditional random variable, that is, 
\[ R_{\pi(j)}(\pi, x) =  \set{ y \in [n] : \exists \, Y \in \mathcal{S}: Y_{\pi(j)} = y \text{ and } \forall s < j, Y_{\pi(s)} = x_{\pi(s)} }, \] 
and let $N_{\pi(j)}(\pi, x) = \size{R_{\pi(j)}(\pi,x)}$ be the size of this range. Note that $R_{\pi(j)}(\pi,x)$, and hence also $N_{\pi(j)}(\pi, x)$, only depends on the first $j-1$ coordinates of $x$ with respect to $\pi$; for $s \ge j$, the values $x_{\pi(s)}$ can be chosen arbitrarily without changing the range of the random variable.

Thus, by~\eqref{eq:uniformentropy}, we can bound the conditional entropy by $H(X_{\pi(j)}| X_{\pi(s)} = x_{\pi(s)}: s < j) \le \log\parens*{N_{\pi(j)}(\pi, x)}$ for all $x \in [n]^{n^2}$.  Substituting this into~\eqref{eq:entropycalc} and reordering the sum gives
\[ \log \, \size{\mathcal{S}} \le \sum\limits_{j=1}^{n^2} \sum\limits_{x\in [n]^{n^2}}\Pr[X=x] \log \parens*{N_{\pi(j)}(\pi, x)} = \sum\limits_{\ell = 1}^{n^2} \E_X \brackets{\log\parens*{N_{\ell}( \pi, x )}}. \]

This bound holds for any permutation $\pi$, and thus it holds when we average over the choice of $\pi$. We sample a uniformly random permutation of $[n^2]$ by choosing a vector $\alpha = (\alpha_\ell)_\ell$ with $\alpha_{\ell}\sim U[0,1]$ for all $1\leq \ell \leq n^2$ and defining $\pi_{\alpha} = \pi$ to be such that $\alpha_{\pi(1)} > \alpha_{\pi(2)} > \dots > \alpha_{\pi(n^2)}$.
We then have 
\begin{align*}
    \log \, \size{\mathcal{S}} &\leq \E_\alpha\brackets*{\sum\limits_{\ell=1}^{n^2}\E_X [\log (N_{\ell}(\pi,x))]} = \sum\limits_{\ell=1}^{n^2}\E_X \brackets*{\E_\alpha [\log (N_{\ell}(\pi,x))]} \\
    &= \sum\limits_{\ell=1}^{n^2}\E_X \brackets*{\E_{\alpha_{\ell}} [\E_{\alpha|\alpha_{\ell}}[\log (N_{\ell}(\pi,x))]]} \leq \sum\limits_{\ell=1}^{n^2}\E_X \brackets*{\E_{\alpha_{\ell}} [\log(\E_{\alpha|\alpha_{\ell}}[N_{\ell}(\pi,x)])]}, 
\end{align*}
where the last inequality follows from Jensen's inequality and the concavity of $y \mapsto \log y$. It therefore suffices to show that, for all $\ell\in [n^2]$ and all $x\in \mathcal{S}$, we have
\begin{equation} \label{eq:numoptions}
    \E_{\alpha_\ell}[\log(\E_{\alpha|\alpha_{\ell}}[N_{\ell}(\pi,x)])] \leq \int_0^1 \log(1 + (r_\ell + c_\ell)t^{d-1} + (n - r_\ell - c_\ell - 1)t^d) \dif t.
\end{equation}

\smallskip

We first estimate the inner expectation $\E_{\alpha|\alpha_{\ell}}[N_{\ell}(\pi, x)] = \E_{\alpha}[N_{\ell}(\pi, x)|\alpha_{\ell}]$. By the linearity of expectation, this is equal to $\sum\limits_{y\in[n]}\mathbb{P} \brackets{y \in R_{\ell}(\pi,x)|\alpha_{\ell}}$.  Unfortunately, it is not straightforward to determine whether or not $y \in R_{\ell}(\pi,x)$, and so we shall instead use a simple necessary condition that we call \emph{availability}.

Recall that for the column $x$ to be orthogonal to the $i$th column of $A$, the pairs $(A(s,i),x_s)$ must be distinct for all $s \in [n^2]$. Therefore, if for some symbol $y \in [n]$ there is some column $i \in [d]$ and previously exposed coordinate $s$ such that $A(s,i) = A(\ell,i)$ and $x_s = y$, we cannot also have $x_\ell = y$. In this case we declare $y$ \emph{unavailable}, and observe that we must have $y \notin R_{\ell}(\pi, x)$.  Otherwise, if there is no such column $i$ and coordinate $s$, we say $y$ is \emph{available}.  We now seek to compute the probability that a symbol $y$ is available.

\smallskip

Fix a symbol $y \in [n]$. If $y$ is the true value of the entry in the $\ell$th coordinate of $x$, then $y$ cannot possibly have been ruled out by the previously exposed entries, and is thus available with probability $1$.

Now suppose $y \in [n]$ is not the true value of $x_\ell$. For each $i \in [d]$, since $x$ is orthogonal to the $i$th column of $A$, there must be a unique entry $s_i(y) \neq \ell$ such that $x_{s_i(y)} = y$ and $A(s_i(y),i) = A(\ell,i)$. In order for $y$ to be available, $\ell$ must be exposed before the entries in the set $S(y) = \{s_i(y) : i \in [d] \}$.

To find the probability of $y$ being available, then, we need to compute the size of $S(y)$. Suppose for distinct columns $1 \le i < j \le d$ we had $s_i(y) = s_j(y)$. It then follows that $A(s_i(y),i) = A(\ell,i)$ and $A(s_i(y),j) = A(\ell,j)$, and thus the $i$th and $j$th columns cannot be orthogonal. Since $A$ is nearly orthogonal, the only possibilities are $i \in \{1,2\}$ and $j = 3$ (by definition, all columns after the third column are orthogonal to all others, and the first two columns are orthogonal by construction).

Therefore $\size{S(y)} = d$, unless either $s_1(y) = s_3(y)$ or $s_2(y) = s_3(y)$. Note that these cannot happen simultaneously, as we have ruled out $s_1(y) = s_2(y)$, and thus in these cases we have $\size{S(y)} = d-1$. There are $r_\ell$ choices of $s \neq \ell$ for which $A(s,1) = A(\ell,1)$ and $A(s,3) = A(\ell,3)$, and hence $r_\ell$ values $y$ for which $s_1(y) = s_3(y)$. By orthogonality of $x$ with the first column of $A$, these values are all distinct. Similarly, there are $c_\ell$ choices for $y$ with $s_2(y) = s_3(y)$. 

\medskip

To summarize, there is one choice of $y$ that is available with probability $1$, there are $r_\ell + c_\ell$ choices of $y$ that are available only if the $\ell$th coordinate is exposed before some fixed set of $d-1$ other coordinates, and the remaining $n - r_\ell - c_\ell - 1$ choices of $y$ are available only if the $\ell$th coordinate precedes some $d$ other coordinates.

A coordinate $s$ is revealed after $\ell$ if $\alpha_s < \alpha_\ell$, which occurs with probability $\alpha_\ell$. Moreover, these events are independent for distinct coordinates, and so the probabilities in the latter two cases are $\alpha_\ell^{d-1}$ and $\alpha_\ell^d$ respectively. This gives
\[ \E_{\alpha|\alpha_{\ell}}[N_{\ell}(\pi,x)] = \sum\limits_{y \in [n]} \Pr \brackets{y \in R_{\ell}(\pi, x)} \le \sum\limits_{y \in [n]} \Pr \brackets{y \text{ is available}} = 1+(r_\ell + c_\ell) \alpha_\ell^{d-1} + ( n - r_\ell - c_\ell - 1) \alpha_\ell^{d}. \]
Since $\alpha_\ell$ is uniformly distributed over $[0,1]$, substituting this into $\E_{\alpha_\ell}[\log(\E_{\alpha|\alpha_{\ell}}[N_{\ell}(\pi, x)])]$ results in~\eqref{eq:numoptions}, completing the proof.
\end{proof}

\subsection{Estimating the integral}

In order to apply Theorem~\ref{thm:numberofextensions}, we need to understand the asymptotics of the bound it provides. In this next lemma, we show how to estimate the integral from the theorem.

\begin{lemma}\label{lem:integral}
Let $2\leq d\leq n$ and $I_d = \int_0^1\log(1+(n-1)t^d) \dif t$. Then
\[ I_d \leq \log \parens*{\frac{n-1}{e^d}} + \frac{d}{(n-1)^{1/d}} + \frac{3}{d (n-1)^{1/d}}. \]
\end{lemma}
\begin{proof}
Set $t_0 = (n-1)^{-1/d}$. Note that $(n-1)t^d < 1$ if and only if $t< t_0$.
    We have 
    \begin{align*}
        I_d &= \int_0^1 \log (1+(n-1)t^d) \dif t \\
        &= \int_0^{t_0} \log (1+(n-1)t^d) \dif t + \int_{t_0}^1 \log ((n-1)t^d) \dif t + \int_{t_0}^1 \log \parens*{1+\frac{1}{(n-1)t^d}} \dif t. 
    \end{align*}
    
    We estimate the three integrals in turn: 
    \[ \int_0^{t_0} \log (1+(n-1)t^d) \dif t \leq \int_0^{t_0} (n-1)t^d \dif t = \frac{n-1}{d+1}t^{d+1} \Big|_0^{t_0} = \frac{t_0}{d+1}, \]
    where for the inequality we use the fact that $\log (1+x) \leq x$ for all $x > -1$, and in the final equality we use $t_0^d = (n-1)^{-1}$,
    \begin{align*}
        \int_{t_0}^1 \log ((n-1)t^d) \dif t &= \int_{t_0}^1 \log(n-1) + d \log t \dif t \\
        &= (1-t_0)\log (n-1) + d\parens{t\log t - t}\Big|_{t_0}^1 \\
        &= (1-t_0)\log(n-1) + (t_0-1)d +t_0\log(n-1)\\
        &= \log(n-1) + (t_0-1)d,
    \end{align*}
    where the penultimate equality again follows from $t_0^d = (n-1)^{-1}$, and
    \[ \int_{t_0}^1 \log \parens*{1+\frac{1}{(n-1)t^d}} \dif t \leq \int_{t_0}^1 \frac{1}{(n-1)t^d} \dif t = -\frac{1}{(n-1)(d-1)} + \frac{t_0}{d-1}. \]
    
Hence, we have
    \begin{align*}
        I_d &\leq \frac{t_0}{d+1} + \log(n-1) + (t_0-1)d -\frac{1}{(n-1)(d-1)} + \frac{t_0}{d-1}\\
        &\leq \log \parens*{\frac{n-1}{e^d}} + \frac{d}{(n-1)^{1/d}} + \frac{3}{d(n-1)^{1/d}},
    \end{align*}
    where we ignore the negative term and bound $\frac{1}{d+1} + \frac{1}{d-1}$ by $\frac{3}{d}$.
\end{proof}

Corollary~\ref{cor:asymptoticextension} now follows easily from Theorem~\ref{thm:numberofextensions} and Lemma~\ref{lem:integral}.

\begin{proof}[Proof of Corollary~\ref{cor:asymptoticextension}]
The lower bound comes from the average number of extensions of a $k$-MOLS, computed in~\eqref{eq:luriakeevash}. For the upper bound, Theorem~\ref{thm:numberofextensions} asserts that the logarithm of the number of extensions of a $k$-MOLS of order $n$ is, in the notation of Lemma~\ref{lem:integral}, at most $n^2 I_{k+2}$. By the lemma, this is bounded by
\[ n^2 \parens*{\log \parens*{\frac{n-1}{e^{k+2}}} + \frac{(k+2)}{(n-1)^{1/{(k+2)}}} + \frac{3}{(k+2)(n-1)^{1/{(k+2)}}}} \le n^2 \parens*{\log \parens*{\frac{n-1}{e^{k+2}}} + \frac{k+4}{(n-1)^{1/{(k+2)}}}}. \]
Since $k$ is fixed as $n$ tends to infinity, this is
\[ n^2 \parens*{\log \parens*{\frac{n-1}{e^{k+2}}} + o(1)} = n^2 \log \parens*{(1 + o(1)) \frac{n-1}{e^{k+2}}} = n^2 \log \parens*{(1 + o(1)) \frac{n}{e^{k+2}}}, \]
giving the desired upper bound.
\end{proof}

Finally, we deduce our upper bound on the number of large sets of mutually orthogonal Latin squares.

\begin{proof}[{Proof of Corollary \ref{cor:numsets}}]
    We can build a $k$-MOLS by starting with the empty $0$-MOLS, and extending it by one Latin square at a time. Theorem~\ref{thm:numberofextensions} bounds the number of possible extensions at each step, and so, in the notation of Lemma~\ref{lem:integral}, we have
    \begin{equation}
        \log \numkmols{k}{n} \le n^2 \sum\limits_{d=2}^{k+1} I_d.
    \end{equation}
    
    We shall prove each part of the corollary by estimating this sum appropriately.

\begin{enumerate}[(i)]
    \item By Lemma~\ref{lem:integral}, we have
    \[ I_d \le \log(n-1) - d + \frac{d+2}{(n-1)^{1/d}}. \]
    Hence, summing over $d$, we obtain
    \[ \sum\limits_{d=2}^{k+1} I_d \le k \log(n-1) - \parens*{\binom{k+2}{2} - 1} + \binom{k+4}{2} (n-1)^{-1/(k+2)}, \]
    from which the bound follows.
    
    \item Rearranging the bound in Lemma~\ref{lem:integral} gives
    \begin{align*}
        I_d &\leq \log(n-1) - d\parens*{1 - (n-1)^{-1/d}} + \frac{3}{d(n-1)^{1/d}} \\
        &\leq \log(n-1) - d \parens*{1 - e^{-\tfrac{\log(n-1)}{d}}} + \frac{3}{d}.
    \end{align*}
    Therefore we have
    \[ \sum\limits_{d=2}^{k+1} I_d \leq k \log(n-1) - \sum\limits_{d=2}^{k+1} d \parens*{1 - e^{-\tfrac{\log(n-1)}{d}}} + \sum\limits_{d=2}^{k+1} \frac{3}{d}. \]
    The second sum, an error term, is at most $3 \log (k+1)$. For the first sum, by making the substitution $x = \frac{d}{\log(n-1)}$, we observe that this is related to the estimation of the integral $\int x \parens*{1 - e^{-1/x}} \dif x$ by the Riemann sum with step size $1/\log(n-1)$. More precisely, we have
    \[ \frac{1}{\log(n-1)} \sum\limits_{d=2}^{k+1} \frac{d}{\log(n-1)} \parens*{1 - e^{-\tfrac{\log(n-1)}{d}}} = \int_{\tfrac{2}{\log(n-1)}}^{\tfrac{k+1}{\log(n-1)}} x \parens*{1 - e^{-1/x}} \dif x + o(1). \]
    Making the necessary substitutions and letting $n$ tend to infinity gives the claimed bound.
    
    \item In this range, we estimate $e^{-1/x}$ to observe that
    \[ x \parens*{1 - e^{-1/x}} = 1 - \frac{1}{2x} + O(x^{-2}), \]
    where the asymptotics are as $x$ tends to infinity.  Hence, when $\beta$ tends to infinity,
    \[ \int_0^{\beta} x \parens*{1 - e^{-1/x}} \dif x = \int_0^{\beta} 1 - \frac{1}{2x} + O(x^{-2}) \dif x = \beta - \tfrac12 \log(\beta) + O(1). \]
    The result then follows by substituting this into the statement of part (ii) with $\beta = \tfrac{k}{\log n}$; since the integrand $x \parens*{1 - e^{-1/x}}$ is bounded and monotone increasing for large $x$, the Riemann sum remains a good approximation of the integral when $\beta \rightarrow \infty$. \qedhere
\end{enumerate}
\end{proof}

\section{Explicit constructions}\label{sec:constructions}

Corollary~\ref{cor:asymptoticextension} establishes the existence of Latin squares with several orthogonal mates.  Given the numerous applications of orthogonal Latin squares, however, it is of great interest to have explicit constructions of such squares.  For instance, in the closely related problem of counting transversals in Latin squares, Taranenko~\cite{taranenko2015transversals} showed that a Latin square of order $n$ can have at most $\parens*{(1+o(1))\frac{n}{e^2}}^{n}$ transversals. Glebov and Luria~\cite{glebov2016maximum} later proved that Taranenko's bound is tight via a probabilistic construction. Recent results of Eberhard et al.~\cite{eberhard2019additive} and Eberhard~\cite{eberhard2017more} give a constructive proof of the theorem of Glebov and Luria, providing explicit examples of Latin squares attaining this bound (in a very precise sense).  They show that the Cayley table of any abelian group $G$ where $\sum\limits_{g \in G} g = 0$ has $\parens*{\tfrac{2 \pi n^2}{\sqrt{e}}+o(1)}\parens*{\frac{n}{e^2}}^n$ transversals. 

To see the relation between transversals and orthogonal mates, observe that the $n$ translates of any transversal in a Cayley table partition the Latin square.  For each such partition into transversals, we can construct $n!$ distinct orthogonal mates by assigning distinct symbols in $[n]$ to the $n$ transversals. The above results thus imply that these Latin squares have at least $\parens*{\sqrt{\tfrac{8 \pi^3 n^5}{e}}+o(1)}\parens*{\frac{n^2}{e^3}}^n$ orthogonal mates. This lower bound is much smaller than the upper bound we would like to match, because this simple argument only counts orthogonal mates of a very special type. Here we describe a construction of MacNeish~\cite{macneish1922} that allows us to significantly improve this bound, even if we still fall slightly short of the true maximum number of orthogonal mates given by Corollary~\ref{cor:asymptoticextension}.
\smallskip

The {\em Kronecker product} of two Latin squares $L_1$ and $L_2$ of order $n_1$ and $n_2$ respectively is the Latin square $L_1\otimes L_2$ of order $n_1n_2$ given by $(L_1\otimes L_2)((i_1,j_1),(i_2,j_2)) = (L_1(i_1,i_2), L_2(j_1,j_2))$. (Of course, the row and column indices and the symbols of $L_1\otimes L_2$ can be seen as elements of $[n_1n_2]$ after fixing an arbitrary bijection from $[n_1] \times [n_2]$.) For a Latin square $L$, we write $L^{\otimes k}$ to denote the $k$-fold product $\underbrace{L\otimes \dots \otimes L}_{k \text{ times}}$.

\begin{proposition}\label{prop:prodtwosquares}
Let $L_1$ and $L_2$ be Latin squares of order $n_1$ and $n_2$ that have $q_1$ and $q_2$ orthogonal mates respectively. Then the number of orthogonal mates of $L_1\otimes L_2$ is at least ${q_1 q_2^{n_1^2}}\frac{(n_1n_2)!}{n_1!(n_2!)^{n_1}}$.
\end{proposition}
\begin{proof}
    We will show how orthogonal mates of $L_1$ and $L_2$ can be combined in several ways to produce orthogonal mates of the product $L_1 \otimes L_2$.  For this, it is again useful to view an orthogonal mate as an ordered partition of $L_1 \otimes L_2$ into disjoint transversals.

    Further observe that $L_1\otimes L_2$ can be partitioned into $n_1^2$ blocks of the form $L_1(i_1, i_2) \otimes L_2$ for $i_1, i_2 \in [n_1]$. Each of these is isomorphic to $L_2$, and thus admits $q_2$ orthogonal mates.
    
    There are $q_1$ orthogonal mates of $L_1$, and thus $\tfrac{q_1}{n_1!}$ \emph{unordered} partitions of $L_1$ into disjoint transversals, say $\{T_1, \hdots, T_{n_1} \}$. In the product $L_1 \otimes L_2$, this partitions the blocks into $n_1$ disjoint sets.
    
    Let $T_j$ be one of the transversals in this decomposition of $L_1$. The corresponding blocks $T_j \otimes L_2 = \{ L_1(i_1,i_2) \otimes L_2: (i_1, i_2) \in T_j \}$ then have all distinct symbols from $[n_1]$ in the first coordinate, and hence cover each symbol in $[n_1] \times [n_2]$ precisely $n_2$ times.  To get a transversal of $L_1 \otimes L_2$, we can choose a transversal in each block $L_1(i_1, i_2) \otimes L_2$ and stitch them together. Furthermore, if we partition each block into transversals, stitching them together gives a partition of $T_j \otimes L_2$ into transversals of $L_1 \otimes L_2$.
    
    There are $q_2^{n_1}$ ways to choose orthogonal mates for each of the $n_1$ blocks in $T_j \otimes L_2$.  Here we keep the ordering, as that tells us which transversals in different blocks should be stitched together.  This gives us an ordered partition of $T_j \otimes L_2$ into $n_2$ transversals of $L_1 \otimes L_2$, and so there are $\tfrac{q_2^{n_1}}{n_2!}$ unordered partitions of this set of blocks into transversals.
    
    Making these choices for each $T_j$, we obtain a total of $\tfrac{q_1}{n_1!} \parens*{\tfrac{q_2^{n_1}}{n_2!}}^{n_1}$ partitions of $L_1 \otimes L_2$ into $n_1n_2$ disjoint transversals, each of which can easily be shown to be distinct.  To obtain an orthogonal mate, we can order these transversals arbitrarily, and thus obtain $q_1 q_2^{n_1^2} \tfrac{(n_1n_2)!}{n_1!(n_2!)^{n_1}}$ mates, as claimed. 
\end{proof}

In particular, this implies that powers of a single Latin square have many orthogonal mates.

\begin{corollary}\label{cor:blowuplatinsquare}
Let $L$ be a Latin square of order $m$ with $q$ orthogonal mates. Then $L^{\otimes k}$ is a Latin square of order $m^k$ with at least $q^{\tfrac{m^{2k}-1}{m^2-1}}$ orthogonal mates.
\end{corollary}
\begin{proof}
    We proceed by induction. The statement is clearly true for $k=1$. Suppose it holds for some $k\geq 1$. Then, by Proposition~\ref{prop:prodtwosquares}, with $L_1 = L^{\otimes k}$, $L_2 = L$, $n_1 = m^k$, $n_2 = m$, $q_1 = \tfrac{m^{2k}-1}{m^2-1}$, and $q_2=q$, we know the number of mates of $L^{\otimes (k+1)}$ is at least 
    \[ q^{\tfrac{m^{2k}-1}{m^2-1}}q^{m^{2k}} \frac{(m^{k+1})!}{(m^k)!(m!)^{m^k}} \geq q^{\tfrac{m^{2(k+1)}-1}{m^2-1}}. \qedhere \]
\end{proof}

If we take $L$ to be the Cayley table of $\mathbb{Z}_3$, then we have $q=6$. The $k$-fold Kronecker product of the Cayley table gives the Cayley table of the product group $\mathbb{Z}_3^k$, which by Corollary~\ref{cor:blowuplatinsquare} has at least $(6^{1/8})^{3^{2k}-1}$ orthogonal mates. In the next corollary, we show that the constant in the base of the exponent can be made arbitrarily large at the cost of having a slightly less explicit construction.

\begin{corollary}\label{cor:largeconst}
For any $C > 0$, there are infinitely many orders $n$ for which we can efficiently produce Latin squares with at least $C^{n^2}$ orthogonal mates.
\end{corollary}
\begin{proof}
    Let $m$ be such that $\frac{m}{2e^3} > C$.  By Corollary~\ref{cor:asymptoticextension}, provided $m$ is sufficiently large, there is a Latin square $L$ of order $m$ with $\parens*{(1 + o(1))\frac{m}{e^3}}^{m^2} > C^{m^2}$ orthogonal mates; we can find such a square with a (finite) exhaustive search. By Corollary~\ref{cor:blowuplatinsquare}, we know that the Latin square $L^{\otimes k}$ of order $n = m^k$ has at least $C^{\tfrac{m^2(m^{2k}-1)}{m^2-1}} \geq C^{n^2}$ orthogonal mates.
\end{proof}

\section{Concluding remarks and open questions} \label{sec:conclusion}
In this paper, by bounding the number of extensions of a set of mutually orthogonal Latin squares, we obtained upper bounds on the number of $k$-MOLS when $k$ grows with $n$. The obvious question is how tight these bounds are --- can we find corresponding lower bounds? The constructions of Donovan and Grannell~\cite{donovan2013transversal}, valid for infinitely many values of $n$ when $k \le \sqrt{n}$, give lower bounds of the form $\log \numkmols{k}{n} = \Omega \parens*{\gamma(k,n) n^2 \log n}$, where $\gamma(k,n) = \max \left\{\tfrac{\log k}{k^2 \log n}, \tfrac{1}{k^4} \right\}$.  This is considerably smaller than our upper bounds in Corollary~\ref{cor:numsets}, and it would be of great interest to narrow the gap. One might hope to extend the lower bounds of Keevash~\cite{keevash2018coloured}, which were tight for constant $k$, but, as he notes in his paper, it is unclear how his methods could be used when $k$ grows.

Aside from the enumeration of $k$-MOLS, there are several other related open problems, and we elaborate on these possible directions of study below.

\paragraph{Orthogonal mates} We have bounded the maximum number of orthogonal mates a Latin square can have, but it is natural to ask if it is typical for a Latin square to have any orthogonal mates at all. Computational results in this direction are given in \cite{bryant2013number}, \cite{egan2016enumeration}, and \cite{mckay2007smalllatinsquares}. His study of squares of small order led van Rees~\cite{vanrees1990subsquares} to conjecture that, as $n\rightarrow\infty$, the proportion of Latin squares without orthogonal mates tends to one. On the other hand, having studied slightly larger orders, Wanless and Webb \cite{wanless2006existence} suggested that the opposite may be true. 

In~\eqref{eq:extensionlowerbound}, we saw that the results of Luria~\cite{luria2017new} and Keevash~\cite{keevash2018coloured} imply that the average Latin square of order $n$ has $\parens*{(1 + o(1)) \frac{n}{e^3}}^{n^2}$ orthogonal mates. Since Theorem~\ref{thm:numberofextensions} shows that this is roughly the maximum number of mates a square can have, this implies that at least $\parens*{(1 + o(1)) \frac{n}{e^2}}^{n^2}$ Latin squares must have an orthogonal mate. Unfortunately, due to the asymptotic error in the base of the exponent, this falls short of resolving the question of whether or not most Latin squares have orthogonal mates.

Some evidence that this may not be straightforward to resolve is provided in~\cite{cavenagh2017notransversals}, where Cavenagh and Wanless showed that, for almost all even $n$, there are at least $n^{(1-o(1))n^2}$ Latin squares of order $n$ without a transversal, let alone an orthogonal mate. However, Ferber and Kwan~\cite{ferberkwan2019} study the analogous question in Steiner triple systems, and show that almost all Steiner triple systems are almost resolvable. In the context of Latin squares, they say (personal communication) that their methods would show that almost all Latin squares have $(1 - o(1))n$ disjoint transversals.  Still, some new ideas would be needed to find the $n$ disjoint transversals that form an orthogonal mate.

\smallskip

In Section~\ref{sec:constructions} we showed, for any given $C>0$, that we can, for infinitely many $n$, construct Latin squares of order $n$ with at least $C^{n^2}$ orthogonal mates. Given the existence of Latin squares with many more, namely $n^{(1 + o(1))n^2}$, orthogonal mates, it is natural to seek better constructions. 
\begin{problem}
Is there an explicit construction of a Latin square of order $n$ with at least $n^{\Omega (n^2)}$ orthogonal mates? 
\end{problem}
\noindent In our product construction in Section~\ref{sec:constructions}, we only considered orthogonal mates consisting of very special kinds of transversals (those built within blocks, using transversals of the two factor squares).  It is likely that these product squares have a much larger number of orthogonal mates, perhaps even close to the maximum possible.

We have also been vague with regards to what we mean by an explicit construction. As is customary in computer science, by {\em explicit} we mean there is an algorithm that constructs the Latin square in question in time polynomial in $n$. One can go further, and call a construction {\em strongly explicit} if each individual entry of the Latin square can be determined in polylogarithmic time. One can verify that our construction in the previous section is indeed strongly explicit. Yet one feels somewhat cheated, as in the first step of the construction we perform an exhaustive search to find an initial Latin square with many orthogonal mates (whose existence is guaranteed by random methods, see Corollary~\ref{cor:asymptoticextension}). It would be desirable to find constructions that are also ``morally explicit'' in the sense that they can be described mathematically, and in particular avoid any initial brute-force search. In this direction, it would be natural to investigate whether the Cayley tables of abelian groups $G$ with $\sum\limits_{g\in G} g = 0$ give examples of such Latin squares (cf. \cite{eberhard2017more, eberhard2019additive}).

\paragraph{Affine and projective planes} As mentioned earlier, $(n-1)$-MOLS of order $n$ correspond to affine, and hence projective, planes of order $n$.  Our knowledge of lower bounds in this setting is even direr; it is conjectured that there are no such systems when $n$ is not a prime power, and believed that there is a unique (up to isomorphism) projective plane when $n$ is a prime. As a step towards proving these conjectures, one could seek to bound the number of affine/projective planes from above, a problem raised by Hedayat and Federer~\cite{hedayat1971embedding}.

In our definition of $\numkmols{k}{n}$, we do not account for isomorphism.  Thus, given a single projective plane, we can permute the symbols within each square of the corresponding $(n-1)$-MOLS to obtain $(n!)^{n-1}$ distinct $(n-1)$-MOLS. This gives a lower bound of $\numkmols{n-1}{n} \ge (n!)^{n-1} = e^{(1 - o(1)) n^2 \log n}$ whenever $n$ is a prime power. We remark that, for certain prime powers $n$, Kantor~\cite{kantor2003commutative} and Kantor and Williams~\cite{kantor2004symplectic} provide algebraic constructions of superpolynomially many non-isomorphic projective planes of order $n$, but this contributes a lower order term in the above bound.

For an upper bound, Corollary~\ref{cor:numsets} yields $\numkmols{n-1}{n} \le e^{\parens*{\tfrac12 + o(1)} n^2 \log^3 n}$. However, since a projective plane corresponds to a maximum possible set of mutually orthogonal Latin squares, it has a very restricted structure, and we can take advantage of this to obtain a better upper bound.  Given a projective plane $\Pi_n$ of order $n$, a subset $H$ of its lines is called a \emph{defining set} if $\Pi_n$ is the unique projective plane containing $H$ --- that is, the lines in $H$ determine the remaining lines in $\Pi_n$. Building on the work of Kahn \cite{kahn1992random}, Boros, Sz\H{o}nyi and Tichler~\cite{boros2005defining} showed that every projective plane admits a small defining set.
\begin{theorem}[{Boros, Sz\H{o}nyi and Tichler, 2005}]\label{thm:definingsets}
Every projective plane of order $n$ (for $n$ sufficiently large) contains a defining set of size at most $22n\log n$.
\end{theorem}
This immediately improves our upper bound.
\begin{corollary}\label{cor:numplanes}
$\numkmols{n-1}{n} \le e^{(22+o(1))n^2\log^2n}$.
\end{corollary}

\begin{proof}
By Theorem \ref{thm:definingsets}, each projective plane of order $n$ contains a set $H$ of $22n\log n$ lines that determine the remaining ones uniquely. Each line is a subset of size $n+1$ of the $n^2+n+1$ points. Thus, there are $\binom{n^2+n+1}{n+1} = e^{(1+o(1))n\log n}$ possible lines and at most $(e^{(1+o(1))n\log n})^{22n\log n}$ possible sets $H$ and hence projective planes of order $n$.

Each $(n-1)$-MOLS corresponds to an affine plane of order $n$ (the $n^2$ cells represent the points, two parallel classes are formed by the rows and columns, and each of the $n-1$ Latin squares labels the lines of one of the remaining parallel classes), and every affine plane has a unique extension to a projective plane (for each parallel class, we extend the lines to a common new point, and add a line at infinity consisting of the new points). In the reverse direction, a projective plane corresponds to at most $n^2 + n + 1$ different affine planes, by choosing the line at infinity. Furthermore, each affine plane corresponds to at most $(n+1)! (n!)^{n+1} = e^{(1 + o(1)) n^2 \log n}$ distinct $(n-1)$-MOLS, since we can permute the parallel classes, and the lines within parallel classes.  This contributes a lower order term, and so we can also bound $\numkmols{n-1}{n} \le e^{(22 + o(1)) n^2 \log^2 n}$.
\end{proof}

It would be of great interest to remove the extra logarithmic factor in the exponent of the upper bound, and thus reduce it log-asymptotically to the lower bound.  
\begin{conjecture} \label{conj:projplanes}
$$L^{(n-1)}(n) = e^{O\left( n^2 \log n\right)}.$$
\end{conjecture}
\noindent As we obtain this many $(n-1)$-MOLS from a single projective plane, this would provide qualitative evidence in favor of the non-existence conjectures concerning projective planes. Note that a stronger result than Conjecture~\ref{conj:projplanes} (and a possible avenue of attack) would be to improve Theorem~\ref{thm:definingsets}, which is not known to be tight.  The best known lower bound for Theorem~\ref{thm:definingsets} is only linear in $n$, and, if every projective plane were to indeed contain a defining set of $O(n)$ lines, that would imply $\numkmols{n-1}{n} = e^{\Theta(n^2 \log n)}$.

\paragraph{Sudoku squares} As mentioned in Section~\ref{sec:gerechte}, Sudoku squares are a special class of gerechte designs of order $n$, where $n = m^2$, with the array partitioned into $m \times m$ subsquares in the natural way. After Golomb~\cite{golomb200611214} asked about the existence of a pair of orthogonal Sudoku squares of order~9 (corresponding to the popular puzzle), systems of {\em $k$ mutually orthogonal Sudoku squares ($k$-MOSS)} have been studied by several authors. This research has primarily sought to determine the largest $k$ for which a $k$-MOSS of order $n$ can exist; we refer the reader to~\cite{bailey2008gerechte, keedwell2007sudoku, keedwell2010constructions, lorch2009linearsudoku, lorch2010combings, pedersen2009MOSLS} for constructions and results in this direction.

One may ask the same counting questions as before for this restricted class of Latin squares, and these are relatively less well-studied. The number of Sudoku squares of order $n$ is known to be $\parens*{(1+o(1))\frac{n}{e^3}}^{n^2}$; the upper bound is shown independently by Luria~\cite{luria2017new} (using entropy) and Berend~\cite{berend2018sudoku} (using Br\'egman's Theorem), while the matching lower bound is due to Keevash~\cite{keevash2018coloured}.

To enumerate $k$-MOSS for fixed $k$, we extend an idea of Keevash, defining a $4$-uniform hypergraph $H$ with vertices $V(H) = \set{x_1,x_2,y_1,y_2,z_1^{(1)}, z_2^{(1)}, \dots, z_1^{(k)},z_2^{(k)}}$ and edges $\set{x_1, x_2, y_1, y_2}$, $\set{x_1, x_2, z_1^{(i)}, z_2^{(i)}}$, $\set{y_1, y_2, z_1^{(i)}, z_2^{(i)}}$, $\set{x_1, y_1, z_1^{(i)}, z_2^{(i)}}$ and $\set{z_1^{(i)}, z_2^{(i)}, z_1^{(j)}, z_2^{(j)}}$ for all $1 \le i < j \le k$.  Letting $H(\sqrt{n})$ be the $(2k+4)$-partite $4$-uniform hypergraph obtained by blowing each vertex up into $\sqrt{n}$ new vertices, it follows that a $k$-MOSS is equivalent to a decomposition of $H(\sqrt{n})$ into copies of $H$.  For fixed $k$, the results of Luria and Keevash show there are $\parens*{(1 + o(1)) \frac{n^k}{e^{\binom{k+3}{2} - 3}}}^{n^2}$ such decompositions.

\medskip

Our results allow us to bound the number of ways of extending a $k$-MOSS by an additional Sudoku square.  Since each cell shares its row (or column) with $\sqrt{n}-1$ other cells from the same subsquare, we have $r_{\ell} = c_{\ell} = \sqrt{n} - 1$ for each $\ell \in [n^2]$.  Applying Theorem~\ref{thm:numberextensions} shows that our upper bounds on $k$-MOSS coincide with our upper bounds on $(k+1)$-MOLS.  In particular, the bound in (a) is once again tight, as it matches the average number of extensions of a $k$-MOSS.

\begin{corollary}\label{cor:numberSudokuextensions}
\begin{enumerate}[(a)]
    \item For all fixed $k$, the maximum number of extensions of a $k$-MOSS of order $n$ to a $(k+1)$-MOSS is $\parens*{(1 + o(1)) \frac{n}{e^{k+3}}}^{n^2}$.
    \item For $k = k(n) \ge 0$, the logarithm of the number of $k$-MOSS is at most
    \begin{enumerate}[(i)]
        \item $\parens*{(k+1) \log n - \binom{k+3}{2} + 3 + (k+1)^2 n^{-1/(k+3)} + o(1) } n^2 \hfill \text{if } k = o(\log n),$
        \item $\parens*{c(\beta) + o(1)} k n^2 \log n \hfill \text{if } k = \beta \log n, \text{ for fixed } \beta > 0,$
        \item $\parens*{\tfrac12 + o(1)} (\log k - \log \log n) n^2 \log^2 n \hfill \text{if } k = \omega( \log n ),$
    \end{enumerate}
    where $c(\beta)$, as in Corollary~\ref{cor:numsets}, is defined to be $1 - \beta^{-1} \int_0^{\beta} x (1 - e^{-1/x}) \dif x$.
\end{enumerate}
\end{corollary}

\begin{proof}
    A $k$-MOSS corresponds to a $NOA(n,k+3)$ with $r_\ell = c_\ell = \sqrt{n} - 1$.  Substituting these parameters into Theorem~\ref{thm:numberextensions}, the logarithm of the number of extensions of a $k$-MOSS is at most
    \begin{align*}
        &n^2 \int_0^1 \log \parens*{1 + 2(\sqrt{n} - 1)t^{k+2} + (n - 2\sqrt{n} + 1) t^{k+3}} \dif t \\
        &= n^2 \int_0^1 \log \parens*{1 + (n-1)t^{k+3}} \dif t + n^2 \int_0^1 \log \parens*{1 + \frac{2(\sqrt{n} - 1)(1 - t)t^{k+2}}{1 + (n-1)t^{k+3}}} \dif t.
    \end{align*}
    The first integral is simply $I_{k+3}$, as evaluated in Lemma~\ref{lem:integral}.
    
    To bound the second integral, observe that
    \begin{align*}
        \int_0^1 \log \parens*{1 + \frac{2(\sqrt{n} - 1)(1 - t)t^{k+2}}{1 + (n-1)t^{k+3}}} \dif t &\le \int_0^1 \frac{2(\sqrt{n} - 1)(1 - t)t^{k+2}}{1 + (n-1)t^{k+3}} \dif t \\
        &\le \int_0^1 \frac{(2 \sqrt{n - 1}) t^{k+2}}{1 + (n-1)t^{k+3}} \dif t \\
        &= \frac{2}{(k+3) \sqrt{n-1}} \log \parens*{1 + (n-1)t^{k+3}} \Big|_0^1 = \frac{2 \log n}{(k+3) \sqrt{n-1}}.
    \end{align*}
    
    Thus, even if we sum up over all $k \in [n]$, the contribution from this second integral is a lower order error term.  Hence our upper bound on the logarithm of the number of extensions of a $k$-MOSS is $n^2 (I_{k+3} + o(1))$, and therefore we obtain the same enumeration as when extending a $(k+1)$-MOLS.
\end{proof}

Aside from the general lower bound of Keevash~\cite{keevash2018coloured}, we are not aware of any lower bounds on the number of $k$-MOSS. It would therefore be interesting to find lower bounds on the number of $k$-MOSS when $k$ grows with $n$. 
\begin{problem}
How tight are the upper bounds in Corollary~\ref{cor:numberSudokuextensions}(b)?  That is, for $k = k(n)$ that grows with $n$, can we show the existence of many distinct $k$-MOSS?
\end{problem}

\bibliographystyle{amsplain}
\bibliography{orthogonal_latin_squares}
\end{document}